\documentclass[12pt]{amsart}
\usepackage{amsmath, amsthm, amssymb}
\usepackage{graphicx}
\usepackage{mathbbol}
\usepackage{txfonts}
\usepackage[usenames]{color}
\usepackage{enumerate}
\usepackage{hyperref}

\newtheorem{theorem}{Theorem}[section]
\newtheorem{theoremalph}{Theorem}

\newtheorem{lemma}[theorem]{Lemma}

\newtheorem{definition}[theorem]{Definition}

\newtheorem{remark}[theorem]{Remark}

\newcommand{\be}{\begin{equation}}

 \newcommand{\ee}{\end{equation}}

\newcommand{\R}{\mathbb{R}}

\newcommand{\diam}{\operatorname{Diam}}

\newcommand{\hm}{{\mathcal H}}
\newcommand{\lm}{{\mathcal L}}

\newcommand{\set}{{\rm{set}}}

\DeclareMathOperator{\Ric}{Ric}

\newcommand{\Lip}{\operatorname{Lip}}

\newcommand{\mass}{{\mathbf M}}

\begin{document}

\title[IF and GH convergence with Ricci bounded below]{Intrinsic flat and Gromov-Hausdorff convergence of manifolds with Ricci curvature bounded below}

\author{Rostislav Matveev}
\address{Max Planck Institute for Mathematics in the Sciences}
\email{matveev@mis.mpg.de}

\author{Jacobus W. Portegies}
\thanks{The authors would like to thank the Max Planck Institute for Mathematics in the Sciences for its hospitality and support.}
\address{Max Planck Institute for Mathematics in the Sciences}
\email{jacobus.portegies@mis.mpg.de}




\begin{abstract}
  We show that for a noncollapsing sequence of closed, connected, oriented
  Riemannian manifolds with Ricci curvature uniformly bounded from
  below and diameter uniformly bounded above, Gromov-Hausdorff convergence
  essentially agrees with intrinsic flat convergence.
\end{abstract}

\maketitle

In this manuscript, we consider the class $\mathcal{M}(n, \Lambda, v, D)$ of $n$-dimensional, closed, connected, 
oriented Riemannian manifolds $M$ with
\begin{equation*}
\Ric_{M} \geq -(n-1) \Lambda, \quad \hm^n(M) \geq v > 0, \quad \diam(M) \leq D,
\end{equation*}
and show that for a sequence of manifolds in this class, Gromov-Hausdorff convergence essentially agrees with intrinsic flat convergence.

The intrinsic flat distance was introduced by Sormani and Wenger
\cite{Sormani-Intrinsic}, and relates to the flat distance as the Gromov-Hausdorff distance relates to the Hausdorff distance.

In general, there are 
 significant differences between 
Gromov-Hausdorff convergence and intrinsic flat convergence: the
intrinsic flat limit may be noncompact, there are sequences
of Riemannian manifolds that do have an intrinsic flat limit but do
not have a Gromov-Hausdorff limit, the intrinsic flat limit is always
rectifiable, etc.. 
However, in the presence of a uniform lower bound on the Ricci
curvature, the theory on the structure of Gromov-Hausdorff limits
developed by Cheeger and Colding \cite{Cheeger-Colding-I,
  Cheeger-Colding-II, Cheeger-Colding-III} suggests that the two
concepts may not differ all that much.

The first result in this direction goes back to Sormani and Wenger
\cite{Sormani-Weak}. 
They show that if a sequence of manifolds $M_i \in \mathcal{M}(n, \Lambda = 0, v, D)$ with \emph{nonnegative} Ricci curvature converges in the Gromov-Hausdorff distance to a metric space $X$, then a subsequence will converge in the intrinsic flat distance to an integral current space $(X, d_X, T)$: a metric space $X$ endowed
with an integral current $T$ (in the sense of Ambrosio-Kirchheim
\cite{Ambrosio-Currents}), that is completely settled 
(meaning $X$ is
exactly the set of positive $n$-dimensional lower density for $T$).
Recently, Munn \cite{Munn-Intrinsic} has obtained a similar result for sequences of manifolds with a uniform, two-sided bound on the Ricci curvature.

Li and Perales \cite{Li-Perales} have proved that for a
sequence of integral current spaces for which the metric spaces are
Alexandrov spaces of nonnegative curvature and have a uniform diameter
upper bound, either the sequence converges in the intrinsic flat distance
to the zero space, or a subsequence converges in both the
Gromov-Hausdorff and the intrinsic flat distance, and the underlying
metric spaces in the limit are the same.

This manuscript extends the results by
Sormani and Wenger \cite{Sormani-Weak} and Munn \cite{Munn-Intrinsic} to sequences of manifolds
with an arbitrary uniform lower bound
on the Ricci curvature and additionally shows that the limiting current is essentially unique, has multiplicity one, 
and has mass equal to the Hausdorff measure on the limiting space.

Before we state our main results more precisely, we need to introduce some notation.
The class $\mathcal{M}(n, \Lambda, v, D)$ is
precompact both in the Gromov-Hausdorff distance and in the intrinsic flat distance \cite{Sormani-Weak}.  We denote the completion
with respect to these distances respectively by $\mathcal{M}^{GH}(n,
\Lambda, v, D)$ and $\mathcal{M}^{IF}(n, \Lambda, v, D)$.

There is an involution $\iota$ acting on the space $\mathcal{M}^{IF}(n, \Lambda, v, D)$ by reverting the orientation of the current.  We
denote by $\mathcal{M}^{IF/\iota}(n, \Lambda, v, D)$ the quotient
metric space obtained from $\mathcal{M}^{IF}(n, \Lambda, v, D)$ by
identifying every integral current space with its image under the
involution.

Our first main theorem is the following.
\begin{theoremalph}
\label{th:FirstMain}
The map $F : \mathcal{M}^{IF/\iota} (n, \Lambda, v, D) \to
\mathcal{M}^{GH}(n, \Lambda, v, D)$ given by
\begin{equation*}
(X, d_X, T)\mapsto X
\end{equation*}
is well-defined and is a homeomorphism. 

\noindent Additionally, for all $(X, d_X, T)
\in \mathcal{M}^{IF}(n, \Lambda, v, D)$,
\begin{enumerate}[(i)]
\item the set of positive lower-density $\set(T)$ of $T$ equals $X$,
\item the mass measure $\|T\|$ equals the Hausdorff measure $\hm^n$ on $X$,
\item the multiplicity of $T$ is equal to $1$, $\hm^n$-a.e.
\end{enumerate}
\end{theoremalph}

This theorem implies the aforementioned results by Sormani and Wenger and Munn.
Yet it also
illustrates that the currents are just going along for the ride: up to
an involution, the currents are uniquely determined by the underlying
metric spaces.

From the fact that the mass measure $\|T\|$ equals the Hausdorff measure $\hm^n$, it follows that it is the limit (in a weak sense) of the Riemannian volume measures of the manifolds in the approximating sequence. Indeed, Cheeger and Colding \cite[Theorem 5.9]{Cheeger-Colding-I} have shown that under Gromov-Hausdorff convergence, the Riemannian volume measures on the manifolds converge to the Hausdorff measure on the limit space.

We believe that the idea of the original proof by Sormani and Wenger
in \cite{Sormani-Weak}
can be used to extend their result to the case of Ricci curvature
bounded below by an arbitrary constant. 
The important
ingredients in their proof, such as the application of a volume estimate by
Colding \cite[Corollary 2.19]{Colding-Ricci} and Perelman's Main Lemma in \cite{Perelman-Manifolds}, 
are applicable to manifolds of almost nonnegative curvature, and
therefore they can be applied after scaling.  To our
knowledge, this was so far unknown.
Our proof will differ from the one by Sormani and Wenger and does not use Perelman's Main Lemma.

Our second main result is a type of local constancy theorem. 
Alternatively, it may be interpreted as stating that the
local top-dimensional homology of the space is isomorphic to
$\mathbb{Z}$.

\begin{theoremalph}
\label{th:SecondMain}
Let $(X, d_X, T) \in \mathcal{M}^{IF}(n, \Lambda, v, D)$. 
For all $q \in X$, and every integral current $S$ on $X$ such
  that $\| \partial S \|(B_t(q)) = 0$ for some $t > 0$, there exists an integer
  $k \in \mathbb{Z}$ such that $S = k T$ on $B_t(q)$.
\end{theoremalph}

The structure of the manuscript is as follows. 
Section~\ref{se:Background} gives a coarse background on integral
currents in the sense of Ambrosio and Kirchheim
\cite{Ambrosio-Currents}, integral current spaces and intrinsic flat
converge as introduced by Sormani and Wenger \cite{Sormani-Intrinsic}
and some of the elements we need from the theory on the structure of
spaces with Ricci curvature bounded below by Cheeger and Colding
\cite{Cheeger-Colding-I, Cheeger-Colding-II}. In our notation, we
generally try to stick to the notation in these articles.

Theorems~\ref{th:FirstMain} and \ref{th:SecondMain} are direct consequences of Theorem~\ref{th:MainTheoremText} in the text.  A crucial ingredient is a link
between zero-dimensional slices and the degree, which we will explain
in Section~\ref{se:Degree}.  Earlier work by Sormani and the second
author \cite{Portegies-Properties} showed that integrals of the flat
distance between lower-dimensional slices of currents can be
controlled by the flat distances between the original currents. This
will imply that the $L^1$-distance between their degrees can be
controlled locally.  In Section~\ref{se:ColdingVolume}, we show that
Colding's volume estimate easily translates into an estimate on the
degree for manifolds $M \in \mathcal{M}(n, \Lambda, v, D)$.  In
Section~\ref{se:MainTheorem} we combine these two ingredients to give
a proof of the main theorem.

\section*{Acknowledgments}

The authors would like to thank the Max Planck Institute for Mathematics in the Sciences for its hospitality and support. The second author would like to thank Raquel Perales for discussions on the results in \cite{Li-Perales}.

\section{Background}
\label{se:Background}
 
In this section, we review integral currents on metric spaces as introduced by Ambrosio and Kirchheim \cite{Ambrosio-Currents}, the intrinsic flat distance introduced by Sormani and Wenger \cite{Sormani-Intrinsic} and some theory on the structure of spaces with Ricci curvature bounded below by Cheeger and Colding \cite{Cheeger-Colding-I, Cheeger-Colding-II}. The prime purpose of this review is to fix notation. We adhere closely to the notation used in these articles, and therefore the reader familiar with these works could probably understand the rest of the manuscript without reading this section.

\subsection{Currents}

Let $X$ be a complete metric space.  For $n \geq 1$, we define the set
$\mathcal{D}^n(X)$ of all $(n+1)$-tuples $(f, \pi_1, \dots, \pi_n)$ of
Lipschitz functions on $X$, where additionally $f$ is required to be
bounded. For $n = 0$, we define $\mathcal{D}^0(X)$ as the set of
bounded Lipschitz functions.  It can be helpful to think of an element
$(f, \pi_1, \dots, \pi_n) \in \mathcal{D}^{n+1}$ as an $n$-form $f d
\pi_1 \wedge \dots \wedge d\pi_n$.

An $n$-dimensional metric functional is a function $T: \mathcal{D}^n(X) \to \mathbb{R}$ such that the map
\begin{equation}
(f, \pi_1, \dots, \pi_n) \mapsto T(f, \pi_1, \dots, \pi_n)
\end{equation}
is subadditive, and positively $1$-homogeneous with respect to the functions $f$ and $\pi_1, \dots, \pi_n$. 
We denote the vector space of $n$-dimensional metric functionals on $X$ by $MF_n(X)$.

The exterior differential $d$ maps $\mathcal{D}^n(X)$ into $\mathcal{D}^{n+1}(X)$ according to
\begin{equation}
d(f, \pi_1, \dots, \pi_n) := (1, f, \pi_1, \dots, \pi_n).
\end{equation}
For $n \geq 1$, the boundary of $T \in MF_n(X)$, is the $(n-1)$-dimensional metric functional denoted by $\partial T$ defined by
\begin{equation}
\partial T (\omega) := T (d \omega), \qquad \omega \in \mathcal{D}^{n-1}(X).
\end{equation}

If $Y$ is another complete metric space, and $\Phi: X \to Y$ is Lipschitz, we define the pullback operator that maps $\mathcal{D}^n(Y)$ to $\mathcal{D}^n(X)$ by
\begin{equation}
\Phi^\# (f, \pi_1, \dots, \pi_n) = (f \circ \Phi, \pi_1 \circ \Phi, \dots, \pi_k \circ \Phi).
\end{equation}
We define the pushforward $\Phi_\# T \in MF_n(Y)$ of $T \in MF_n(X)$ by 
\begin{equation}
\Phi_\# T (\omega) := T(\Phi^\# \omega), \qquad \omega \in \mathcal{D}^n(Y).
\end{equation}

For $T \in MF_n(X)$ and $\omega = (g, \tau_1, \dots, \tau_k) \in \mathcal{D}^n(X)$, with $k \leq n$, we define the restriction $T \llcorner \omega \in MF_{n-k}(X)$ by
\begin{equation}
T \llcorner \omega (f, \pi_1, \dots, \pi_{n-k}) := T(f g, \tau_1, \dots, \tau_k, \pi_1, \dots, \pi_{n-k}).
\end{equation}
We say that $T \in MF_n(X)$ has finite mass if there exists a finite Borel measure $\mu$ on $X$ such that for all $(f, \pi_1, \dots,\pi_n) \in \mathcal{D}^n(X)$,
\begin{equation}
|T(f, \pi_1, \dots, \pi_n)| \leq \prod_{i=1}^n \Lip(\pi_i) \int_X |f| \, d\mu,
\end{equation}
where $\Lip(\pi_i)$ denotes the Lipschitz constant of $\pi_i$. 
Moreover, the minimal measure satisfying this bound is called the mass of $T$ and is denoted by $\|T\|$.
When $T$ has finite mass, it can be uniquely extended to a function on $(n+1)$-tuples $(f, \pi_1, \dots, \pi_n)$ for which $f$ is merely bounded Borel, and $\pi_1, \dots, \pi_n$ are Lipschitz.

An $n$-dimensional current $T$ is an $n$-dimensional metric functional with additional properties. 
From the definition by Ambrosio and Kirchheim \cite[Definition 3.1]{Ambrosio-Currents}, immediately stronger properties may be derived. 
We choose to only phrase the stronger properties. The space of $n$-dimensional currents forms a Banach space, with respect to the mass norm $\mass(T) = \|T\|(X)$. We denote the Banach space by $\mass_n(X)$.
Every $T \in \mass_n(X)$ satisfies
\begin{enumerate}[(i)]
\item $T$ is multilinear in $(f, \pi_1, \dots, \pi_n)$, and whenever $f$ and $\pi_1$ are both bounded and Lipschitz, 
\begin{equation*}
T(f, \pi_1, \dots, \pi_n) + T( \pi_1 , f, \dots, \pi_n) = T( 1, f\pi_1, \dots, \pi_n),
\end{equation*}
and
\begin{equation*}
T( f , \psi_1( \pi), \dots, \psi_n(\pi)) = T( f \det \nabla \psi(\pi) , \pi_1, \dots, \pi_n),
\end{equation*}
where $\psi = (\psi_1, \dots, \psi_n) \in [C^1(\R^n)]^n$ and $\nabla \psi$ is bounded;
\item The following continuity property is satisfied
\begin{equation*}
\lim_{i \to \infty} T(f^i, \pi_1^i, \dots, \pi_k^i) = T(f, \pi_1, \dots, \pi_k)
\end{equation*}
whenever $f^i - f \to 0$ in $L^1(X, \|T\|)$ and $\pi_j^i \to \pi_j$ pointwise in $X$ with uniformly bounded Lipschitz constant $\Lip(\pi_j^i) \leq C$;
\item The following locaility property holds: $T(f, \pi_1, \dots, \pi_n) = 0$ if $\{ f \neq 0\} = \cup_i B_i$ where $B_i$ are Borel and $\pi_i$ is constant on $B_i$.
\end{enumerate}
We say that a sequence of currents $T_i \in \mass_n(X)$ converges weakly to $T \in \mass_n(X)$ if for all $\omega \in \mathcal{D}^n(X)$,
\begin{equation}
\lim_{i \to \infty} T_i (\omega) = T( \omega).
\end{equation}
The mass of open sets is lower-semicontinuous under weak convergence, that is for $O \subset X$ open, and $T_i$ converging weakly to $T$,
\begin{equation}
\liminf_{i\to\infty} \|T_i\|(O) \geq \|T\|(O).
\end{equation}

A very important example of an $n$-dimensional current on the Euclidean space $\R^n$ is given by the current induced by a function $g \in L^1(\R^n)$ which we denote by $\llbracket g \rrbracket$ and is defined by
\begin{equation}
\llbracket g \rrbracket ( f, \pi_1, \dots, \pi_n) := \int_{\R^n} g f d\pi_1 \wedge \dots \wedge d\pi_n = \int_{\R^n} g f \det(\nabla \pi) \, dx.
\end{equation}
We say that a current $T \in \mass_n(X)$ is normal if $\partial T \in \mass_{n-1}(X)$. 

A subset $S \subset X$ is called countably $\hm^n$-rectifiable if there are compact sets $K_i \subset \R^n$ and Lipschitz functions $f_i: K_i \to X$ such that 
\begin{equation}
\hm^n\left(  S \backslash \cup_{i=1}^\infty f_i(K_i) \right) = 0.
\end{equation}
We say $T \in \mass_n(X)$ is rectifiable if $\|T\|$ is concentrated on a countably $\hm^n$-rectifiable set and vanishes on $\hm^n$-negligible Borel sets. We call $T$ integer rectifiable if for all $\phi \in \Lip(X, \R^n)$ and all open $O \subset X$ it holds that $\phi_\#(T\llcorner O) = \llbracket \theta\rrbracket$ for some $\theta \in L^1(\R^n, \mathbb{Z})$. Finally, the collection of integral currents will consist of all integer rectifiable currents that are also normal. We will denote this collection by $I_n(X)$ and in this manuscript, we will only deal with integral currents.

We denote by $\omega_n$ the (Lebesgue) volume of the unit ball in $\R^n$.
For a Borel measure $\mu$ on $X$, we define respectively the $n$-dimensional lower and upper density of $\mu$ in $x\in X$ by
\begin{equation}
\Theta_{n*} (\mu, x) := \liminf_{r \downarrow 0} \frac{\mu(B_r(x))}{\omega_n r^n}, 
\qquad \Theta_n^*(\mu, x) := \limsup_{r \downarrow 0} \frac{\mu(B_r(x))}{\omega_n r^n}.
\end{equation}
If these values coincide, we call the common value the $n$-dimensional density of $\mu$ in $x$ and we denote it by $\Theta_n(\mu,x)$.
We define $\set(T) \subset X$ as
\begin{equation}
\set(T) := \{ x \in X \, | \, \Theta_{n*}(\|T\|, x) > 0 \}.
\end{equation}
For an integer rectifiable current $T$, the mass $\|T\|$ is always concentrated on $\set(T)$ and $\set(T)$ is rectifiable.

Integer rectifiable currents allow for a parametric representation. 
It is a simple consequence of Lusin's theorem and \cite[Theorem 4.5]{Ambrosio-Currents} that if $T$ is an $n$-dimensional integer rectifiable current, there exist a sequence of compact sets $K_i$, numbers $\theta_i \in \mathbb{N}$ and bi-Lipschitz functions $f_i : K_i \to E$ such that $f_i(K_i) \cap f_j(K_j) = \emptyset$ for $i \neq j$, and
\begin{equation}
\label{eq:parametric}
T = \sum_{i = 1}^\infty \theta_i {f_i}_\# \llbracket \chi_{K_i} \rrbracket \quad \text{and} \quad \sum_{i=1}^\infty \theta_i \mass({f_i}_\# \llbracket \chi_{K_i} \rrbracket ) = \mass(T).
\end{equation}

An $n$-dimensional oriented Riemannian manifold $M$ naturally induces a current (on its geodesic metric space), that we will also denote by $\llbracket M \rrbracket$, given by integration of $\omega \in \mathcal{D}^{n}$ over $M$,
\begin{equation}
\llbracket M \rrbracket (\omega) = \int_M \omega = \int_M \langle \omega, \tau \rangle \, d\hm^n,
\end{equation}
where $\tau$ is a (unit) orienting $n$-vector field. In this case, the mass of $\llbracket M \rrbracket$ equals the Riemannian volume.

The intrinsic representation of rectifiable currents by Ambrosio and Kirchheim \cite[Theorem 9.1]{Ambrosio-Currents} shows that at least in some sense, this formula holds for any integer rectifiable current. 
More precisely, if $Z$ is a $w^*$-separable dual space (i.e. $Z = G^*$ for a separable Banach space $G$), and $T$ is an integer rectifiable current on $Z$, then there exists a countably $\hm^n$-rectifiable set $Y\subset Z$, a Borel function $\theta_T: Y \to \mathbb{N}$ (which we call the multiplicity of $T$) with $\int_Y \theta_T \, d\hm^n < \infty$ and an orientation $\tau$ of $Y$ such that
\begin{equation}
\label{eq:RepresentationCurrent}
T(f , \pi_1, \dots, \pi_n) = \int_Y f(z) \theta_T(z) \langle d\pi_1 \wedge \dots \wedge d\pi_n, \tau \rangle \, d\hm^n(z),
\end{equation}
for $(f, \pi_1, \dots, \pi_n) \in \mathcal{D}^n(Z)$. 
We sometimes write $T = \llbracket Y, \theta_T, \tau \rrbracket$. 
The multiplicity $\theta_T$ corresponds to the $\theta_i$ in the parametric representation (\ref{eq:parametric}), in the sense that for $\|T\|$-a.e. $x \in f_i(K_i)$, $\theta(x) = \theta_i$.

Moreover, the mass of $T$ satisfies
\begin{equation}
\|T\| = \lambda \theta_T \hm^n \llcorner Y,
\end{equation}
for a Borel function $\lambda: Y \to [c(n), C(n)]$ (the area factor) that is bounded away from zero and infinity by constants that only depend on the dimension. 

It is true that we have not defined the objects appearing in (\ref{eq:RepresentationCurrent}). 
For a precise definition and formulation see \cite{Ambrosio-Currents}. 
For the purpose of the paper we just would like to stress the analogy with the formula for a current induced by a Riemannian manifold.

Additionally, the representation formula makes clear that if there is another current $S$ supported on a subset of $Y$, it holds that $S = T \llcorner (b / \theta_T)$ for a Borel function $b:Y \to \mathbb{Z}$. 
We will denote the ratio $(b/ \theta_T)$ by $\Delta S / \Delta T$.

On $Y$, a version of the Lebesgue differentiation theorem is still valid. 
By \cite[Theorem 5.4]{Ambrosio-Rectifiable} and the remark following it, if $A \subset Y$ is Borel, then for $\hm^n$-a.e. $x \notin A$,
\begin{equation}
\Theta_n( \| T \llcorner \, A  \| , x ) = \Theta_n(\|T\|\llcorner A, x) =  0,
\end{equation}
while on the other hand for every Borel function $g:Y \to \R$, for $\hm^n$-a.e. $x \in Y$,
\begin{equation}
\Theta_n(\|T \llcorner g\|, x) = \lambda(x) \theta_T(x) g(x).
\end{equation}

A very useful technique in dealing with currents on metric spaces is called slicing. 
For the purpose of this manuscript, we only need zero-dimensional slices. 
For $T \in I_n(X)$, a Lipschitz map $\Phi: X \to \R^n$ and points $x \in \R^n$, the slices $\langle T, \Phi, x \rangle \in I_0(X)$ are characterized by the property
\begin{equation}
\int_{\R^n} \langle T, \Phi, x \rangle \psi(x) \, dx = T \llcorner (\psi \circ \Phi) \, d \Phi, \qquad \text{for all } \psi \in C_c(\R^n),
\end{equation}
from which it follows that for any bounded Borel function $f$ on $X$,
\begin{equation}
\int_{\R^n} \langle T, \Phi, x \rangle (f) \, dx = T (f d\Phi).
\end{equation}
By \cite[Theorem 9.7]{Ambrosio-Currents}, if $X$ is in addition a $w^*$-separable dual space, and $T= \llbracket Y, \theta_T, \tau \rrbracket$, for a rectifiable set $Y \subset X$, the slices are very easy to interpret. Indeed, for $\lm^n$-a.e. $x \in \R^n$, $\Phi^{-1}(x)\cap Y$ contains at most finitely many points and
\begin{equation}
\langle T, \Phi, x \rangle = \sum_{p \in \Phi^{-1}(x)\cap Y} a_p \theta_T(p) \delta_p, 
\end{equation} 
for some choice of $a_p \in \{ - 1, 1 \}$. 
Finally, for instance by \cite[Theorem 5.7]{Ambrosio-Currents} if $S \in I_n(Z)$, with $\hm^n(\set(T) \backslash \set(S)) = 0$, for $\lm^n$-a.e. $x \in \R^n$, 
\begin{equation}
\langle S, \Phi, x \rangle = \langle T, \Phi, x \rangle \llcorner \frac{\Delta S}{\Delta T}.
\end{equation}

Finally, we note that a separable space $Y$ can always be isometrically embedded into a $w^*$-separable Banach space. 
Indeed, if $y_i$ ($i=1, 2, \dots$) is a dense sequence in $Y$, we may embed $Y$ into the $w^*$-separable Banach space $L^\infty(\{ y_i \}_i)$ by a Kuratowski embedding $I: Y \to L^\infty(\{y_i\}_i)$ given by
\begin{equation}
(I(y))(y_i) := d(y, y_i) - d(y, y_1).
\end{equation}

\subsection{Integral current spaces and intrinsic flat convergence}

Let $S, T \in I_n(X)$.
The flat distance between $S$ and $T$ in $X$ is defined as
\begin{equation}
d_F^X(S, T) := \inf\{ \mass(U) + \mass(V) \, | \, S - T = U + \partial V, \, U \in I_n(X), \, V \in I_{n+1}(X) \}.
\end{equation}

As briefly mentioned in the introduction, an $n$-dimensional integral current space $(X, d_X, T)$ is a pair of a metric space $(X, d_X)$, which is not necessarily complete, and a current $ T \in I_n(\bar{X})$ on the completion of $X$. 
Additionally, by convention, it is assumed that the current is completely settled, that is $X = \set(T)$. 

The intrinsic flat distance between two integral current spaces $(X, d_X, T)$ and $(Y, d_Y, S)$ is given by 
\begin{equation}
\begin{split}
d_{\mathcal{F}}((X, d_X, T), (Y, d_Y, S)) 
&:= \inf \Big\{ d_F^Z(\phi_\# T, \psi_\# S ) \, | \, Z \text{ complete metric space} \\
&\qquad \qquad \phi: X \to Z, \psi: Y \to Z \text{ isometric} \Big\}
\end{split}
\end{equation}
If $d_\mathcal{F}( (X, d_X, T), (Y, d_Y, S)) = 0$, this implies that there exists a current-preserving isometry $\phi: X \to Y$, that is an isometry such that $\phi_\# T = S$. 
Note that without the convention $X = \set(T)$ (or a similar condition), this would certainly not be the case in general. 

We will denote the metric space of (equivalence classes of) $n$-dimensional integral current spaces with the intrinsic flat distance by $\mathcal{M}_n^{IF}$.

There is an involution $\iota$ acting on $\mathcal{M}^{IF}_n$, given by 
\begin{equation}
\iota ( X, d_X, T) := (X, d_X, - T).
\end{equation}
We may endow the quotient space $\mathcal{M}^{IF}_n/\iota$ by the quotient distance:
\begin{equation}
\begin{split}
d_{\mathcal{F}/\iota} (M, N) :=  \inf\{ d_\mathcal{F} (M, \iota^{\epsilon_1} M_1) &+ d_\mathcal{F}(M_1, \iota^{\epsilon_2} M_2) + \dots + d_\mathcal{F}(M_{n-1},\iota^{\epsilon_n} N) \\
& \, | \, M_i \in \mathcal{M}^{IF}_n, \epsilon_i \in \{0,1\} \}
\end{split}
\end{equation}
In general, such a definition only yields a pseudometric. However, in this special case, the quotient distance is indeed a distance, and is in fact given by
\begin{equation}
d_{\mathcal{F}/\iota} (M_1, M_2 ) = \min( d_\mathcal{F}(M_1, M_2) , d_\mathcal{F}(M_1, \iota(M_2)) ).
\end{equation}
We denote the metric space $\mathcal{M}_n^{IF}/\iota$ endowed with this quotient metric by $\mathcal{M}_n^{IF/\iota}$.

\subsection{The structure of spaces with Ricci curvature bounded below}
\label{suse:CheegerColding}

In \cite{Cheeger-Colding-I, Cheeger-Colding-II, Cheeger-Colding-III}, Cheeger and Colding study the structure of metric spaces that arise as the Gromov-Hausdorff limits of manifolds with Ricci curvature uniformly bounded from below. 

Cheeger and Colding consider both noncollapsed and collapsed limit spaces. 
From the point of view of intrinsic flat convergence, the collapsed case is trivial as in that case the approximating sequence of Riemannian manifolds converges in the intrinsic flat distance to the zero integral current space. 
We therefore consider only noncollapsed limit spaces, for which the results by Cheeger and Colding are much stronger. 

We borrow the following definitions. 
If $X \in \mathcal{M}^{GH}(n, \Lambda, v, D)$, and $p \in X$, we say $p \in \mathcal{R}_{\epsilon, \delta}$ if and only if
\begin{equation}
d_{GH}(B_r(p), B_r(0) ) < \epsilon r, \quad \text{for all } r < \delta.
\end{equation}
We further define the regular set $\mathcal{R}$ by
\begin{equation}
\mathcal{R} := \bigcap_{\epsilon > 0} \bigcup_{\delta > 0} \mathcal{R}_{\epsilon, \delta}.
\end{equation}
In words, $p \in X$ is regular if for all $\epsilon > 0$ there exists a $\delta > 0$ such that $B_r(p)$ is $\epsilon$-close to a Euclidean ball at every scale smaller than $\delta$.

Cheeger and Colding show that in the noncollapsed case, the Hausdorff codimension of the complement $X \backslash \mathcal{R}$ is at least $2$.
They use this to prove in \cite[Corollary 3.9 and 3.10]{Cheeger-Colding-II} a (local) connectedness result. 
More precisely, for all $q_1, q_2 \in \mathcal{R}$, and every $\epsilon, \sigma > 0$, there is a $\delta > 0$ such that there is a path in $\mathcal{R}_{\epsilon, \delta}$ of length smaller than $d(q_1, q_2) + \sigma$ connecting $q_1$ and $q_2$. 
It follows immediately that when $q \in X$, $q_1, q_2 \in B_t(q) \cap \mathcal{R}$, for every $\epsilon > 0$ there is a path in $\mathcal{R}_{\epsilon, \delta}$ connecting $q_1$ and $q_2$ that remains inside $B_t(q)$. 
Indeed, there is a $\sigma > 0$ such that $q_1, q_2 \in B_{t- 3\sigma}(q)$. 
Next, there exists a $q_3 \in \mathcal{R}$ with $d(q_3, q) < \sigma$ by the Bishop-Gromov estimate and the fact that $\mathcal{R}$ has full measure. 
By the above, there exists a $\delta > 0$ such that there are paths in $\mathcal{R}_{\epsilon, \delta}$ from $q_1$ to $q_3$ and from $q_3$ to $q_2$ respectively, both of length smaller than $t - \sigma$, as for $k = 1, 2$,
\begin{equation}
d(q_k, q_3) \leq d(q_k , q) + d(q, q_3) < t - 3 \sigma + \sigma = t - 2\sigma. 
\end{equation}
These paths are contained in $B_t(q)$ and by concatenating them we obtain a path from $q_1$ to $q_3$ in $B_t(q) \cap \mathcal{R}_{\epsilon, \delta}$. 

\section{Degree estimate}
\label{se:Degree}

For an integral current $T \in I_n(X)$ on a complete metric space $X$
and a Lipschitz $\Psi: X \to \R^n$ the pushforward 
$\Psi_\# T \in I_n(\R^n)$ is represented by a unique BV
function, by a
representation theorem due to Ambrosio and Kirchheim 
\cite[Theorem 3.7]{Ambrosio-Currents}. The theorem also ensures
that the mass measure of the boundary of $\Psi_\# T$ equals the total
variation of the distributional derivative of the representing
function.
We call such a function {\it the degree} of $\Psi$ with respect to $T$.

\begin{definition}
Let $X$ be a complete metric space and let $T \in I_n(X)$. 
Let $\Psi:X \to \R^n$ be a Lipschitz map. 
We define the degree of $\Psi$ with respect to $T$ as the (unique) function $\deg(T, \Psi , .) \in \mathrm{BV}(\R^n)$, taking values in $\mathbb{Z}$, that satisfies
\begin{equation*}
\Psi_\# T = \llbracket \deg(T,\Psi,.) \rrbracket.
\end{equation*}
\end{definition}

If $T$ is a current induced by an oriented Riemannian manifold, then the degree defined above indeed corresponds to the usual topological degree. 

By the Ambrosio-Kirchheim slicing theorem \cite[Theorem 5.7]{Ambrosio-Currents} and \cite[Eq. (5.18)]{Ambrosio-Currents} the degree can be evaluated by
\begin{equation}
\label{eq:DegreeEqualities}
\begin{split}
\deg( T, \Psi, x) &= \langle \Psi_\# T, \mathrm{Id}, x \rangle (1) \\
&= \Psi_\# \langle T, \Psi, x \rangle (1) \\
&= \langle T, \Psi, x \rangle(1).
\end{split}
\end{equation}
This observation is very useful in light of the estimates by Sormani and the second author that show that integrals of flat distances between slices of two currents can be controlled by the flat distance between the full currents. 
This translates to an $L^1$ estimate on the difference between the degrees. 

\begin{lemma}
\label{le:L1DegConv}
Let $Z$ be a complete metric space and let $T^k \in I_n(Z)$, $k=1,2$.
Let $\Phi^k: Z \to \R^n$ be such that every component $\Phi^k_j$ is $1$-Lipschitz. 
Then
\begin{equation}
\label{eq:L1IneqDegree}
\begin{split}
\| \deg( T^1, \Phi^1, .)& - \deg(T^2, \Phi^2, .) \|_{L^1(\R^n)} 
\leq d_F( T^1 , T^2 )\\
& \quad 
+ 2 \sum_{j=1}^n \|\Phi^1_j - \Phi^2_j\|_\infty \left( \mass(T^2) + \mass(\partial T^2) \right).
\end{split}
\end{equation}
\end{lemma}

\begin{proof}
We use the estimate by Sormani and the second author \cite[Proposition 4.17]{Portegies-Properties}
\begin{equation}
\begin{split}
\int_{\R^n} & d_F (\langle T^1, \Phi^1, x \rangle, 
\langle T^2, \Phi^2, x \rangle ) \, dx 
\leq 
d_F ( T^1, T^2 ) 
\\ & \quad + 2 \sum_{j=1}^n \| \Phi^1_j - \Phi^2_j \|_\infty \left[\mass(T^2) + \mass(\partial T^2)\right].
\end{split}
\end{equation}
From this inequality, we conclude the estimate (\ref{eq:L1IneqDegree}) by using the link between slices and the degree explained in (\ref{eq:DegreeEqualities}). Indeed, for $\lm^n$-a.e. $x\in \R^n$,
\begin{equation}
\begin{split}
| \deg(T^1, \Phi^1, x) - \deg(T^2 , \Phi^2, x) | 
&\leq \left| \langle T^1, \Phi^1, x \rangle(1) - \langle T^2, \Phi^2, x \rangle (1) \right|\\
& \leq d_F(\langle T^1, \Phi^1, x \rangle, \langle T^2, \Phi^2, x \rangle).
\end{split}
\end{equation}
\end{proof}

\section{A consequence of Colding's volume estimate}
\label{se:ColdingVolume}

Throughout this section we use the notation $B_r(0)$ to denote the open ball of radius $r$ centered at $0$ in the Euclidean space $\R^n$.
Let $X$ be a metric space, $x \in X$, and a radius $R > 0$, such that 
\begin{equation}
d_\mathrm{GH} (B_R(x), B_R(0)) < \epsilon.
\end{equation}
Let $\{x_1, \dots, x_n\} \subset X$ and consider the map $\Phi_{x, x_1, \dots,x_n}$ given by 
\begin{equation}
\Phi_{x, x_1, \dots, x_n} (y) = \left( d_M( x_1, y) - d_M(x_1, x) ,\dots  ,d_M(x_n, y)- d_M(x_n, x) \right).
\end{equation}
We call $\Phi_{x, x_1, \dots, x_n}$ an $(\epsilon,R)$-chart around $x$ if there exist a metric space $Z$ and isometric embeddings $\phi: B_R(x) \to Z$ and $\psi: B_R(0) \to Z$, such that
\begin{equation}
d^Z_H(\phi(B_R(x)), \psi(B_R(0))) < \epsilon,
\end{equation}
and 
\begin{equation}
d_H^Z(\phi(\{x, x_1, \dots, x_n\}), \psi(\{0, Re_1, \dots, Re_n\})) < \epsilon,
\end{equation}
where $\{e_i\}$ is the standard orthonormal frame in $\R^n$, and $d_H^Z$ denotes the Hausdorff distance in $Z$.

Note that every coordinate function of $\Phi_{x, x_1, \dots, x_n}$ is Lipschitz with Lipschitz constant $1$.

In \cite[Section 2]{Colding-Ricci}, Colding shows the following result, that we phrase as a lemma.
\begin{lemma}[Colding \cite{Colding-Ricci}]
\label{le:ColdingVolume}
Let $\eta > 0$. There exist $\Lambda = \Lambda(\eta, n) > 0$, $R = R(\eta, n) > 1$,  and $\epsilon = \epsilon (\eta, n) > 0$ with the following property. For every $n$-dimensional complete Riemannian manifold $M$ with $\Ric_{M} \geq -(n-1) \Lambda$, with $p \in M$, such that $d_\mathrm{GH}(B_R(p), B_R(0) ) < \epsilon$, and every $(\epsilon,R)$-chart $\Phi$ around $p$, it holds that
\begin{equation}
\lm^n\left(B_1(0) \backslash \Phi(B_1(p))\right) < \eta^n/3
\end{equation}
and
\begin{equation}
\lm^n(\Phi(B_1(p))) \leq \hm^n(B_1(p)) < \omega_n + \eta^n/3.
\end{equation}
\end{lemma}

From the above lemma, we can derive that on a large set, the $(\epsilon, R)$-chart $\Phi$ is locally one-to-one.
\begin{lemma}
\label{le:BoundMultiplicity}
Let $\eta > 0$. 
There exist $\Lambda = \Lambda (\eta, n)$, $R = R(\eta, n)$ and $\epsilon = \epsilon(\eta, n)$ with the following property. 
For every $n$-dimensional oriented complete Riemannian manifold $M$ with $\Ric_M \geq -(n-1) \Lambda$, $p \in M$ such that 
\begin{equation*}
d_{\mathrm{GH}} (B_R(p), B_R(0)) < \epsilon,
\end{equation*}
and every $(\epsilon,R)$-chart $\Phi$ around $p$, the set 
\begin{equation*}
G = \left\{ x \in B_1(0) \, \middle| \, \hm^0(\Phi^{-1}(x) \cap B_1(p) ) = 1\right\}
\end{equation*}
satisfies
\begin{equation*}
\lm^n(G) > \omega_n - \eta^n.
\end{equation*}
\end{lemma}
\begin{proof}
We choose the constants $\Lambda, R$ and $\epsilon$ as in Lemma~\ref{le:ColdingVolume}, so that 
\begin{equation}
\omega_n - \eta^n/3 < \lm^n( \Phi(B_1(p) ) ) \leq \hm^n(B_1(p)) < \omega_n + \eta^n/3.
\end{equation}
Since every component of the map $\Phi$ is $1$-Lipschitz, the Jacobian $J\Phi$ of $\Phi$ is bounded, $|J \Phi| \leq 1$. 
We apply the coarea formula and find
\begin{equation}
\begin{split}
\hm^n( B_1(p) ) &\geq \int_{B_1(p)} |J \Phi| \, d\hm^n  \\
&= \int_{\R^n} \hm^0 \left( B_1(p) \cap \Phi^{-1} (x) \right) dx \\
&\geq \lm^n(\Phi(B_1(p)) ).
\end{split}
\end{equation}
Therefore,
\begin{equation}
\lm^n \left( \left\{ x \in \Phi(B_1(p)) \, \middle| \, \hm^0( B_1(p)\cap \Phi^{-1}(x) ) \geq 2 \right\}\right) < 2 \eta^n /3
\end{equation}
and consequently
\begin{equation}
\lm^n \left( G \right) > \omega_n - \eta^n.
\end{equation}
\end{proof}

\begin{lemma}
\label{le:DegreeManifold}
Let $\eta>0$. 
There exist $\Lambda = \Lambda(n)$, $R = R(\eta,n)$ and $\epsilon = \epsilon(\eta,n)$ such that if $M$ is an $n$-dimensional oriented complete Riemannian manifold with $\Ric_{M} \geq - (n-1)\Lambda $, $p \in M$, and
\begin{equation*}
d_{\mathrm{GH}}( B_R(p), B_R(0) ) < \epsilon,
\end{equation*}
then for every $(\epsilon, R)$-chart $\Phi$ around $p$ and
for every $x \in B_{1- \sigma}(0)$,
\begin{equation*}
\deg(T \llcorner B_1(p), \Phi, x) = \deg(T \llcorner B_1(p), \Phi, 0) = \pm 1,
\end{equation*}
where $\sigma = \sigma(\eta,n)$ is defined by
\begin{equation*}
\omega_n - \lm^n(B_{1-\sigma}(0)) = \eta^n.
\end{equation*}
\end{lemma}

\begin{proof}
We may without loss of generality assume that $\eta^n \leq \omega_n/4 =: \bar{\eta}^n$. 

Note that $\Phi$ is a $\sigma$-Gromov-Hausdorff approximation, for $R$ large enough and $\epsilon$ small enough, only depending on $n$. 

We choose $R=R(\eta,n)$ and $\epsilon = \epsilon(\eta, n)$ such that $\Phi$ is an $\sigma$-Gromov-Hausdorff approximation, and moreover
\begin{equation*}
\Lambda_{\ref{le:DegreeManifold}}(n) = \Lambda_{\ref{le:BoundMultiplicity}}( \bar{\eta}, n), \qquad  R_{\ref{le:DegreeManifold}}(\sigma, n) \geq R_{\ref{le:BoundMultiplicity}}(\bar{\eta}, n), \qquad \epsilon_{\ref{le:DegreeManifold}}(\sigma, n) \leq \epsilon_{\ref{le:BoundMultiplicity}}(\bar{\eta}, n),
\end{equation*}
where the subscripts indicated the lemma in which the constants are introduced.

Consequently, $\Phi(\partial B_1(p)) \cap B_{1-\sigma}(0) = \emptyset$ and $\deg(T\llcorner B_1(p), \Phi, .)$ is constant on $B_{1-\sigma}(0)$. 
By Lemma~\ref{le:BoundMultiplicity} and our choice of $\sigma$, 
\begin{equation}
\lm^n\left( \left\{ x \in B_{1-\sigma}(0) \,\middle| \, \hm^0(\Phi^{-1}(x) \cap B_1(p) ) = 1 \right\} \right) > 0.
\end{equation}
Therefore,
\begin{equation}
|\deg(T \llcorner B_1(p), \Phi, 0)| = 1.
\end{equation}
\end{proof}

\section{Proof of the main theorems}
\label{se:MainTheorem}

Theorems~\ref{th:FirstMain} and \ref{th:SecondMain} in the introduction are implied by the following theorem.

\begin{theorem}
\label{th:MainTheoremText}
Let $M_i \in \mathcal{M}(n, K, v, D)$ ($i = 1, 2, \dots$) converge in the Gromov-Hausdorff distance to the (compact) metric space $X$. 
We assume without loss of generality that $M_i$ and $X$ are isometrically embedded in a common $w^*$-separable Banach space $Z$.

Then, 
\begin{enumerate}[(i)]
\item \label{item:setIsX} a subsequence of the associated integral current spaces $M_i$ converges in the flat distance to $(X, d_X , T)$, with $\set(T) = X$,
\item \label{item:TisHa} the mass $\|T\|$ equals $\hm^n$, the Hausdorff measure on $X$,
\item \label{item:multiplicity} $T$ has multiplicity one $\hm^n$-a.e.,
\item \label{item:homology} for every $q \in X$, every $t > 0$, and every $S \in I_n(X)$ with $\|\partial S\|(B_t(q)) = 0$, there exists an integer $k \in \mathbb{Z}$ such that $S = k T$ on $B_t(q)$.
\end{enumerate}
\end{theorem}

Before we prove Theorem~\ref{th:MainTheoremText}, we explain how it implies the theorems in the introduction. 
Certainly, Theorem~\ref{th:SecondMain} is a direct consequence of part (\ref{item:homology}) in the above theorem. 

The map $F:\mathcal{M}^{IF/\iota}(n, \Lambda, v, D) \to \mathcal{M}^{GH}(n, \Lambda, v, D)$ given by
\begin{equation*}
(X, d_X, T) \to X
\end{equation*}
considered in Theorem~\ref{th:FirstMain} is well-defined as 
Sormani and Wenger have shown that if the intrinsic flat distance between two integral current spaces is zero, there is a (current-preserving) isometry between the spaces \cite{Sormani-Intrinsic}. Therefore, the map 
 does not depend on the choice of representative. Moreover, by item (\ref{item:setIsX}) of Theorem~\ref{th:MainTheoremText}, $X$ is indeed a compact metric space. 

The map $F$ is one-to-one by (\ref{item:homology}). 
A one-to-one continuous map from a compact to a Hausdorff space is automatically a homeomorphism. So the only non-trivial part of Theorem~\ref{th:FirstMain} left to show is the continuity of $F$.

If $(X_i, d_{X_i} , T_i) \in \mathcal{M}^{IF/\iota}(n, \Lambda, v, D)$ converge in the intrinsic flat distance to an integral current space $(X, d_X, T)$, there are Riemannian manifolds $M_i\in \mathcal{M}^{IF/\iota}(n, \Lambda, v, D)$ such that the associated integral current spaces $\Lbrack M_i \Rbrack$ converge in the intrinsic flat distance to $(X, d_X, T)$ while 
\begin{equation}
d_\mathcal{F}((X_i, d_{X_i}, T_i), \Lbrack M_i \Rbrack ) \to 0.
\end{equation}
By the Gromov compactness and embedding theorems, we may isometrically embed $F(\Lbrack M_i \Rbrack)$ into a common metric space, in which they converge in the Hausdorff distance to a compact metric space $Y$. It suffices to show that $Y$ is isometric to $X$. 

Theorem~\ref{th:MainTheoremText} shows that $\Lbrack M_i \Rbrack \to (Y, d_Y, S)$ for some $S \in I_n(Y)$.
Hence, 
\begin{equation}
d_\mathcal{F}((X, d_X, T), (Y, d_Y, S)) = 0
\end{equation}
and thus $Y$ is isometric to $X$. 

We will now prove Theorem~\ref{th:MainTheoremText}.

\begin{proof}
By the Ambrosio-Kirchheim compactness theorem \cite[Theorem 5.2]{Ambrosio-Currents}, a subsequence of the associated currents $T_i$ converge in the weak sense to an integral current $T$ without boundary.
A theorem by Wenger \cite[Theorem 1.4]{Wenger-Flat} implies that the $T_i$ converge to $T$ in the flat distance in $Z$ as well. 

Let $p \in X$ be a regular point, that is $p \in \mathcal{R} \subset X$. (The definition of the sets $\mathcal{R}$ and $\mathcal{R}_{\epsilon, \delta}$ are as in \cite{Cheeger-Colding-II}, see also Section~\ref{suse:CheegerColding}).
Let $\eta > 0$.
Let $\Lambda(n)$, $\epsilon(\eta, n)$ and $R(\eta, n)$ be as in Lemma~\ref{le:DegreeManifold}.
Since $p \in \mathcal{R}$, there exists a number $0 < \delta < \sqrt{ \Lambda / K }$ such that $p \in \mathcal{R}_{\epsilon / R, \delta R}$.
In other words, for all $r < \delta$,
\begin{equation}
d_{GH}( B_{r R}(p), B_{r R}(0) ) < \epsilon r.
\end{equation}

We will now show that for a subsequence, we may localize the flat convergence to balls.
Since $M_i \to X$ in the Hausdorff distance in $Z$, there exists a sequence $p_i \in M_i$ such that $p_i \to p$.
By the proof of \cite[Lemma 4.1]{Sormani-Arzela}, see also \cite[Theorem 4.16]{Portegies-Properties}, for yet another subsequence, for $\lm^1$-a.e. radius $0 < r < \delta$, the currents restricted to balls of radius $r$ converge, that is
\begin{equation}
d_F ( T_i \llcorner B_r(p_i), T \llcorner B_r(p) ) \to 0. 
\end{equation}

We choose $p^1, \dots, p^n \in X$ such that 
\begin{equation}
\Phi = \Phi_{p, p^1, \dots, p^n}
\end{equation}
is an $(\epsilon r, r R)$-chart around $p$.

In order to prove (\ref{item:TisHa}), we will lift this $(\epsilon, R)$-chart $\Phi$ to charts on the manifolds in the approximating sequence. By the results in the previous section, we have good control of the degree of these charts, and Lemma~\ref{le:L1DegConv} allows us to pass this control to the limit. Estimates on the degree will immediately imply density estimates.

We argue as follows. Since $M_i \to X$ in the Hausdorff distance, there exist $p^j_i$, $j = 1, \dots, n$ such that $p^j_i \to p^j$ as $i \to \infty$, such that $\Phi^i := \Phi_{p^i, p^i_1, \dots, p^i_n}$ is an $(\epsilon,R)$-chart around $p_i$ for $i$ large enough. 
The maps $\Phi^i$ converge to $\Phi$ uniformly by the triangle inequality.

By Lemma~\ref{le:L1DegConv}, $\deg( T_i \llcorner B_r(p_i), \Phi^i, .) \to \deg( T \llcorner B_r(p),\Phi,.)$ in $L^1(\R^n) $ as $i \to \infty$. 
By this convergence result and Lemma~\ref{le:DegreeManifold}, for all $ x \in B_{(1-\sigma) r}(0)$,
\begin{equation}
\deg(T \llcorner B_r(p), \Phi, x ) = \deg(T \llcorner B_r(p), \Phi, 0 ) = \pm 1.
\end{equation}
In particular,
\begin{equation}
\begin{split}
\|T\|(B_r(p)) &\geq \int_{B_{(1-\sigma)r}(0)} |\deg(T \llcorner B_r(p), \Phi, x)| \, dx \\
&\geq \omega_n r^n - \eta^n r^n.
\end{split}
\end{equation}
Moreover, by lower-semicontinuity of the mass measure under weak convergence,
\begin{equation}
\|T\|(B_r(p)) \leq \liminf_{i \to \infty} \|T_i\|(B_r(p_i)) \leq \omega_n r^n + \eta^n r^n /3.
\end{equation}

In particular, $p \in \set(T)$, and since $\eta>0$ was arbitrary, the density $\Theta_n(\|T\|,p)$ of $\|T\|$ in $p$ is $1$.
Since $\hm^n(X \backslash \mathcal{R}) = 0$, in fact $\hm^n = \|T\|$, which shows (\ref{item:TisHa}).

By the results by Cheeger and Colding however, we know that $\hm^n$ satisfies the Bishop-Gromov estimate, so that in particular, at \emph{every} point in $X$ the density of $\hm^n (= \|T\|)$ is (strictly) positive. 
Consequently, $\set(T) = X$, which finishes the proof of (\ref{item:setIsX}).

Our next objective is to show (\ref{item:multiplicity}), namely that $T$ has multiplicity one. To this end, we use that the measure of the set $\{|\deg(T\llcorner B_r(p), \Phi,.)| =1\}$ is very close to $\|T\|(B_r(p))$. 
This can only happen if most points in the image of $\Phi$ have exactly one pre-image, where the multiplicity is one.

More precisely, since
\begin{equation}
\label{eq:DefinitionGoodSetG}
\begin{split}
(\omega_n - \eta^n)r^n 
&= \int_{B_{(1-\sigma)r}(0)} |\deg(T \llcorner B_r(p) , \Phi, x) | \, dx \\
&\leq \int_{B_{(1-\sigma)r}(0)} \mass( \langle T\llcorner B_r(p), \Phi, x \rangle ) \, dx \\
&\leq \|T\|(B_r(p)) \leq (\omega_n + \eta^n/3 )r^n,
\end{split}
\end{equation}
if we define the ``good'' set $G \subset B_{(1-\sigma)r}(0)$ by
\begin{equation}
G := \left\{ x \in B_{(1-\sigma)r }(0) \, \middle| \, \mass(\langle T\llcorner B_r(p), \Phi, x \rangle ) = 1 \right\},
\end{equation}
it follows that
\begin{equation}
\lm^n( B_{(1-\sigma) r} (0)\backslash G ) < 2 \eta^n r^n.
\end{equation}
For $\lm^n$-a.e. $x \in G$, there is a unique $p_x \in \set(T)\cap B_r(p)$ such that $\Phi(p_x) = x$ and the multiplicity $\theta_T(p_x) = 1$. Therefore, again using that $\|T\|(B_r(p)) \leq \omega_n r^n+\eta^n r^n /3 $, we find
\begin{equation}
\|T\| ( \{ x \in B_r(p) \, | \, \theta_T(x) \neq 1  \} ) < 4 \eta^n r^n.
\end{equation}
This shows that in every $p \in \mathcal{R}$, the density of the set where the multiplicity of $T$ is larger than $1$ equals zero. 
Since $\hm^n( X \backslash \mathcal{R})= 0$, $\theta_T \equiv 1$. This shows (\ref{item:multiplicity}).

Finally, we need to show (\ref{item:homology}). Let therefore $q \in X$ and $S \in I_n(X)$ such that $\|\partial S \|(B_t(q)) = 0$, for some $t > 0$. Let $p \in \mathcal{R} \cap B_t(q)$ and $r>0$ be as above, with the additional assumption that $r < d(p, X \backslash B_t(q))$. 
The representation theorem for integer rectifiable currents \cite[Theorem 9.1]{Ambrosio-Currents} implies that there is a Borel function $\Delta S / \Delta T: X \to \mathbb{Z}$ such that
\begin{equation}
S = T \llcorner \frac{\Delta S}{\Delta T}.
\end{equation}
We will prove (\ref{item:homology}) by showing that on a large set $\Delta S / \Delta T$ is equal to the ratio of two degrees, which in turn are constant.
The next few lines will be devoted to the choice of a good representative for $\Delta S/ \Delta T$.

We first introduce an ``average" version, based on a majority vote:
\begin{equation}
\left( \frac{\Delta S} {\Delta T}\right)_{p,r} := \arg \max_{k \in \mathbb{Z}} \|T\| \left( B_r(p) \cap \left\{ \frac{\Delta S}{\Delta T} = k \right\} \right),
\end{equation}
where, if the maximum is not unique, we give preference to the smallest absolute value of $k$, and if this gives no decision, the positive value. However, this is quite irrelevant. 

By the Lebesgue differentiation theorem, for $\hm^n$-a.e. $q_1 \in X$, 
\begin{equation}
\label{eq:LebDiff}
\frac{\Delta S}{\Delta T} (q_1) := \lim_{r \to 0} \left( \frac{\Delta S}{\Delta T} \right)_{q_1,r}.
\end{equation}
For convenience, we will from now on only work with the precise representative of $\Delta S / \Delta T$ which we define to be the right-hand-side of (\ref{eq:LebDiff}) if this limit exists, and $0$ otherwise.

By the characterization of slices \cite[Theorems 5.7 and 9.7]{Ambrosio-Currents}, for $\lm^n$-a.e. $x \in G$, the slice $\langle T, \Phi, x \rangle$ is just a signed delta-measure,
\begin{equation}
\langle T \llcorner B_r(p), \Phi, x \rangle = \pm \delta_{p_x},
\end{equation}
while 
\begin{equation}
\langle S \llcorner B_r(p), \Phi, x \rangle = \frac{\Delta S}{\Delta T}(p_x) \langle T \llcorner B_r(p), \Phi, x \rangle.
\end{equation}
We apply both sides to the function identically equal to $1$, and conclude that 
\begin{equation}
\frac{\Delta S}{\Delta T} ( p_x ) = \frac{\deg(S \llcorner B_r(p), \Phi, x)}{ \deg(T \llcorner B_r(p), \Phi, x)}.
\end{equation}

If we inspect the proof of Lemma~\ref{le:DegreeManifold}, we see that the constants $\Lambda, R$, and $\epsilon$ were chosen such that $\Phi$ is a $\sigma r$-Gromov-Hausdorff approximation. Hence, both $\deg(T \llcorner B_r(p), \Phi, .)$ and $\deg(S \llcorner B_r(p), \Phi, .)$ are constant on the ball $B_{(1-\sigma) r}(0)$ and we can find a compact subset $\tilde{G} \subset G$, $\lm^n(\tilde{G}) > (\omega_n - 3 \eta^n) r^n$  such that \emph{every} $x \in \tilde{G}$ has exactly one pre-image $p_x$ under $\Phi$  in $\set(T) \cap B_r(p)$, and for this $p_x$ 
\begin{equation}
\label{eq:QuotientDegrees}
\frac{\Delta S}{\Delta T} ( p_x ) = \frac{\deg(S \llcorner B_r(p), \Phi, 0)}{ \deg(T \llcorner B_r(p), \Phi, 0)}.
\end{equation}

It follows that
\begin{equation}
\begin{split}
\|T\|\left( \frac{ \Delta S}{\Delta T} \neq \left(\frac{\Delta S }{\Delta T}\right)_{p,r} \right)
&\leq \|T\|(B_r(p) \backslash \Phi^{-1}(\tilde{G})) \\
&\leq \|T\|(B_r(p)) - \int_{\tilde{G}} |\deg(T \llcorner B_r(p), \Phi, x)| \, dx\\
&\leq \omega_n r^n + \eta^n r^n / 3 - \omega_n r^n + 3 \eta^n \\
&\leq 4 \eta^n r^n.
\end{split}
\end{equation}
Therefore, for $\eta$ small enough, only depending on the dimension, the average $(\Delta S / \Delta T)_{p,r}$ is jointly continuous in $(p,r)$ on the domain $p \in \mathcal{R}_{\epsilon/  R, \delta R} \cap B_t(q)$, $0 < r < \min(\delta, d_X(p, X \backslash B_t(q)) )$. 
In particular, $\Delta S/ \Delta T$ is continuous on $\mathcal{R}_{\epsilon / R, \delta R} \cap B_t(q)$.

However, by a result by Cheeger and Colding \cite[Corollary 3.9 and 3.10]{Cheeger-Colding-II} (see also Section~\ref{suse:CheegerColding}), for all $q_1, q_2 \in B_t(q) \cap X$, there is a $\delta$ such that $q_1$ and $q_2$ lie in the same component of $\mathcal{R}_{\epsilon / R, \delta R} \cap B_t(q)$. 
Therefore, $\Delta S / \Delta T$ is constant on $B_t(p)$. This shows (\ref{item:homology}).
\end{proof}

\begin{remark}
To obtain the necessary control on the degree, we use that an $(\epsilon, R)$-chart $\Phi$ is a Gromov-Hausdorff approximation. 
It is possible to conclude such control under weaker assumptions. A lot of information can be extracted from the existence of a map $\Psi: X \to \R^n$ on the limit space $X$, for which every component is $1$-Lipschitz, and such that the excess $\|T\|(B_r(p)) - T(\chi_{B_r(p)} d\Psi)$ is small. 
We have decided to not present this argument in the manuscript, as it is considerably longer, and at this point, it is unclear whether there are applications in which the easier argument by a Gromov-Hausdorff approximation cannot be applied.
\end{remark}

\bibliographystyle{plain}
\bibliography{SlicVol}

\begin{thebibliography}{10}

\bibitem{Ambrosio-Currents}
Luigi Ambrosio and Bernd Kirchheim.
\newblock Currents in metric spaces.
\newblock {\em Acta Math.}, 185(1):1--80, 2000.

\bibitem{Ambrosio-Rectifiable}
Luigi Ambrosio and Bernd Kirchheim.
\newblock Rectifiable sets in metric and {B}anach spaces.
\newblock {\em Math. Ann.}, 318(3):527--555, 2000.

\bibitem{Cheeger-Colding-I}
Jeff Cheeger and Tobias~H. Colding.
\newblock On the structure of spaces with {R}icci curvature bounded below. {I}.
\newblock {\em J. Differential Geom.}, 46(3):406--480, 1997.

\bibitem{Cheeger-Colding-II}
Jeff Cheeger and Tobias~H. Colding.
\newblock On the structure of spaces with {R}icci curvature bounded below.
  {II}.
\newblock {\em J. Differential Geom.}, 54(1):13--35, 2000.

\bibitem{Cheeger-Colding-III}
Jeff Cheeger and Tobias~H. Colding.
\newblock On the structure of spaces with {R}icci curvature bounded below.
  {III}.
\newblock {\em J. Differential Geom.}, 54(1):37--74, 2000.

\bibitem{Colding-Ricci}
Tobias~H. Colding.
\newblock Ricci curvature and volume convergence.
\newblock {\em Ann. of Math. (2)}, 145(3):477--501, 1997.

\bibitem{Li-Perales}
Nan Li and Raquel Perales.
\newblock On the sormani-wenger intrinsic flat convergence of alexandrov
  spaces.
\newblock {\em 1411.6854 [math.MG]}, 2014.

\bibitem{Munn-Intrinsic}
Michael Munn.
\newblock Intrinsic flat convergence with bounded ricci curvature.
\newblock {\em arXiv:1405.3313v2 [math.MG]}, 2015.

\bibitem{Perelman-Manifolds}
G.~Perelman.
\newblock Manifolds of positive {R}icci curvature with almost maximal volume.
\newblock {\em J. Amer. Math. Soc.}, 7(2):299--305, 1994.

\bibitem{Portegies-Properties}
J.~Portegies and C.~Sormani.
\newblock Properties of the {I}ntrinsic {F}lat {D}istance.
\newblock {\em arXiv:1210.3895 [math.DG]}, 2015.

\bibitem{Sormani-Arzela}
Christina Sormani.
\newblock Intrinsic {F}lat {A}rzela-{A}scoli theorems.
\newblock {\em arXiv:1402.6066 [math.MG]}, 2014.

\bibitem{Sormani-Weak}
Christina Sormani and Stefan Wenger.
\newblock Weak convergence of currents and cancellation.
\newblock {\em Calc. Var. Partial Differential Equations}, 38(1-2):183--206,
  2010.
\newblock With an appendix by Raanan Schul and Wenger.

\bibitem{Sormani-Intrinsic}
Christina Sormani and Stefan Wenger.
\newblock The intrinsic flat distance between {R}iemannian manifolds and other
  integral current spaces.
\newblock {\em J. Differential Geom.}, 87(1):117--199, 2011.

\bibitem{Wenger-Flat}
Stefan Wenger.
\newblock Flat convergence for integral currents in metric spaces.
\newblock {\em Calc. Var. Partial Differential Equations}, 28(2):139--160,
  2007.

\end{thebibliography}

\end{document}